\documentclass[12pt]{article}

\usepackage[a4paper,top=28mm,bottom=28mm,inner=30mm,outer=30mm]{geometry}

\usepackage{amssymb,amsfonts,amsthm, mathtools, latexsym,amsmath,hhline,array,longtable,enumitem}
\usepackage{pdfsync,color, comment,colortbl, mathrsfs,stmaryrd,cite,graphicx,lscape}
\usepackage{tikz, tkz-graph, tkz-berge}
\usepackage{caption, subcaption, float, blindtext, multicol}
\usepackage{nicematrix, listings}

\tikzset{main node/.style={circle,fill=black,draw,minimum size=2mm,inner sep=0pt}, }
\usetikzlibrary{positioning,fit,calc}

\usepackage{ifpdf}
\ifpdf 
\usepackage[colorlinks=true, citecolor=black, linkcolor=black, urlcolor=black]{hyperref} 
\fi

\newcommand{\B}{\mathcal{B}}
\renewcommand{\P}{\mathcal{P}}
\newcommand{\D}{\mathcal{D}}

\usepackage{stmaryrd}

\renewcommand{\mod}{\operatorname{mod} \;}

\newtheorem{formula}{}[section]
\newtheorem{definition}[formula]{Definition}

\newtheorem{remark}[formula]{Remark}

\newtheorem{theorem}[formula]{Theorem}

\newtheorem{example}[formula]{Example}

\title{Switching graphs and designs}

\author{Aida Abiad
\thanks{Department of Mathematics and Computer Science, 
Eindhoven University of Technology, The Netherlands (a.abiad.monge@tue.nl)}
\thanks{Department of Mathematics and Data Science of Vrije Universiteit Brussel, Belgium}
\and 
Dean Crnkovi\'c 
\thanks{Faculty of Mathematics, University of Rijeka, Croatia (deanc@math.uniri.hr)}
\and
Louka Peters
\thanks{Department of Mathematics: Analysis, Logic and Discrete Mathematics, Ghent University, Belgium (loukapeters99@gmail.com)}
\and
Andrea \v{S}vob
\thanks{Faculty of Mathematics, University of Rijeka, Croatia (asvob@math.uniri.hr)}
}

\date{}

\begin{document}

\maketitle

\begin{abstract}
Switching methods can be seen as certain local transformations that do not alter their basic parameters of a combinatorial structure. Efforts have been devoted in the literature to relate and unify the switching theories for codes and designs, and also for Hadamard matrices and graphs. The combinatorial structures we consider in this paper are graphs and designs. We show an extension of known switching method for constructing 2-designs to divisible designs, and then provide some examples of its application. Moreover, we prove several equivalences between switching methods for graphs and designs, and as a byproduct, we obtain a new switching method to obtain 2-designs. 

\medskip

\end{abstract}

%%%%%%%%%%%%%%%%%%%%%%%%%%%%%%%%%
\section{Introduction}
%%%%%%%%%%%%%%%%%%%%%%%%%%%%%%%%%

Multiple local operations of combinatorial structures (such as \linebreak Hadamard matrices, graphs, codes or designs) that leave the basic parameters unaltered, have been widely used in the literature under the name of \emph{switching}. For such switching methods to work, the combinatorial object has to satisfy some conditions.

For the case of such structures being graphs, multiple switching operations for the construction of \emph{cospectral graphs} (graphs with the same spectrum) are known in the literature. Godsil and McKay \cite{Godsil1982} introduced one of the first and most fruitful switching methods (for short, GM-switching). Other switching methods to construct cospectral graphs have recently been introduced in the literature, such as the method found by Wang, Qiu and Hu (for short, WQH-switching) \cite{wang2019cospectral}.
The goal of such graph switching methods is to construct cospectral graphs, which is useful for understanding what graph properties cannot be detected by the spectrum (see e.g. \cite{h1996,hs1995,bch2015,ABBCGV2022}), and which provides new insights to a famous conjecture in this area due to Haemers \cite{which} (``almost all graphs are determined by their spectrum''). Other applications of the graph switching methods are the construction of new strongly regular graphs, or its use for providing a negative answer to Haemers' conjecture for certain graph classes (see e.g.\ \cite{abiad2016switched,cioba,Kubota2016,johnson,haemers2010graphs}).

Regarding switching methods for designs, Denniston \cite{denniston1982enumeration} used a method called switching ovals for a construction of symmetric (25,9,3) designs. A similar idea was used by Orrick in \cite{orrick2008switching}. The switching using Pasch configurations, therefore called Pasch switch, was used in  \cite{gibbons1976computing, grannell2001pasch} for a construction of new Steiner triple systems from known ones. Furthermore, {\"O}sterg{\aa}rd \cite{O2012} introduced a switching for codes and Steiner systems. In \cite{jungnickel1991sdp}, Jungnickel and Tonchev used maximal arcs for a transformation of quasi-symmetric designs that leads to a construction of new designs, that are still quasi-symmetric. This transformation can also be described as switching. In \cite{crnkovic2022switching}, the second and fourth author introduced a switching method (for short, CS-switching) that can be applied to any 2-design having a set of blocks satisfying certain properties.

Attention has been devoted to studying the link between switching methods for different combinatorial objects. One common feature of the switching methods is that they define a local transformation with variants for codes, designs, and related combinatorial structures. Through the definition of this transformation, termed switching, a variety of earlier methods and results for different structures can be unified; a first attempt is the work by \"Ostergård \cite{O2012}. In particular, in \cite{O2012} the author establishes links between switching methods for combinatorial designs and codes. Recently, an equivalence between two switching methods to construct inequivalent Hadamard matrices by Orrick \cite{orrick2008switching} and a switching method for constructing cospectral graphs by Godsil and McKay \cite{Godsil1982} has been shown by the first and the third author in \cite{abiad2024switching}. In \cite{CrEganSvob}, a definition of switching was introduced that incorporates various known types of switching. Further, this switching was applied to extend Orrick’s switching methods to Butson Hadamard and complex Hadamard matrices. In this paper, we show some new relationships and equivalences between switching methods for designs and graphs. In particular, we show an extension of CS-switching for designs, so that it can no longer only be used for 2-designs. We continue with a generalization of switching an oval for designs, and apply that to Menon designs from Hadamard matrices. For this application, we expand the existing connection between switching an oval and switching a closed quadruple to GM-switching. Afterwards, we examine the effect of CS-switching on the incidence graph of a 2-design, and find some conditions for which CS-switching can be made equivalent to GM-switching. Remaining on the topic of CS-switching and graph switching, we consider the necessary conditions for WQH-switching of the incidence matrix of a 2-design to deliver again an incidence matrix of a 2-design. From this we define a new switching method for 2-designs and show how it extends a switching result by Ihringer and Munemasa \cite{ihringer2019new}. Finally, we provide several examples of divisible designs that can be obtained from the new switching method.

%%%%%%%%%%%%%%%%%%%%%%%%%%%%%%%%%
\section{Preliminaries}\label{sec:prelim}
%%%%%%%%%%%%%%%%%%%%%%%%%%%%%%%%%

Let us start by giving some background information on the link between graphs and designs via incidence graphs, and on the known switching methods to construct cospectral graphs and 2-designs.

%%%%%%%%%%%%%%%%%%%%%%%%%%%%%
\subsection{Graphs and designs}
%%%%%%%%%%%%%%%%%%%%%%%%%%%%%

An \emph{incidence structure} is an ordered triple ${\mathcal{D}}=({\mathcal{P}},{\mathcal{B}},{\mathcal{I}})$ 
where ${\mathcal{P}}$ and ${\mathcal{B}}$ are non-empty disjoint sets and 
$ {\mathcal{I}}\subseteq {\mathcal{P}}\times {\mathcal{B}}$.
The elements of the set ${\mathcal{P}}$ are called points, the
elements of the set ${\mathcal{B}}$ are called blocks and
$\mathcal{I}$ is called an incidence relation.
If $|{\mathcal{P}}|=|{\mathcal{B}}|$, then the incidence structure is
called symmetric. The incidence
matrix of an incidence structure is a $v \times b$ matrix
$[m_{ij}]$ where $v$ and $b$ are the numbers of points and blocks
respectively, such that $m_{ij} = 1$ if the point $P_i$ and the block 
$x_j$ are incident, and $m_{ij} = 0$ otherwise. An isomorphism
from one incidence structure to another is a bijective mapping of
points to points and blocks to blocks which preserves incidence.

A \emph{$t$-$(v,k,\lambda)$ design} is a finite incidence structure
${\mathcal{D}}=({\mathcal{P}},{\mathcal{B}},{\mathcal{I}})$ satisfying the
following requirements:
\begin{enumerate}
  \item $|{\mathcal{P}}|=v$,
  \item every element of ${\mathcal{B}}$ is incident with exactly
  $k$ elements of ${\mathcal{P}}$,
  \item every $t$ elements of ${\mathcal{P}}$ are incident with exactly
$\lambda$ elements of ${\mathcal{B}}$.
\end{enumerate}

In a $t$-$(v,k,\lambda )$ design every point is incident with exactly $r=\frac{\lambda (v-1)}{k-1}$ blocks, and $r$ is called the replication number of a design.
A 2-design is also called a block design. If $v=b$, a design is called \textit{symmetric}. In a symmetric design, each two blocks meet in exactly $\lambda$ points.

An incidence structure with $v$ points and the constant block size $k$ is a 
(group) divisible design with parameters $(v,k, \lambda_1, \lambda_2, m,n)$ 
whenever the point set can be partitioned into $m$ classes of size $n$, 
such that two points from the same class are incident with exactly $\lambda_1$ common blocks, and two points from different classes are 
incident with exactly $\lambda_2$ common blocks. A divisible design $D$ is said to be symmetric (or to have the dual property) if the dual of $D$ is a 
divisible design with the same parameters as $D$. 

Every design has an incidence graph, making it a very useful tool for our research. 

\begin{definition}\cite{haemers2011matrices}
    Given any design \(\D = (\P,\B,I)\), the \emph{incidence graph} $IG(\D)$ is a bipartite graph with vertex set \(\P \cup \B\) and an edge between \(P \in \P\) and \(B \in \B\) if and only if \((P,B) \in I\).
\end{definition}

The incidence graph of a \(t\)-\((v,k,\lambda)\) design is \((r,k)\)-regular and has the property that every \(t\) vertices of \(\P\) have \(\lambda\) common neighbours.
 If $M$ is the incidence matrix of a design $\D$, then the adjacency matrix of $IG(\D)$ is given as follows
   \[
    \begin{pmatrix} 
        O   &  M\\
       M^T  &  O 
    \end{pmatrix},
    \]
where $O$ denotes a zero matrix.

%%%%%%%%%%%%%%%%%%%%%%%%%%%%%%
\subsection{Switching graphs}\label{Chap:SwitchGraphs}
%%%%%%%%%%%%%%%%%%%%%%%%%%%%%%

%%%%%%%%%%%%%%%%%%%%%%%%%%%%%%
\subsubsection{GM-switching}
%%%%%%%%%%%%%%%%%%%%%%%%%%%%%%
The most well-known method to construct cospectral graphs is the so-called Godsil-McKay switching (or GM-switching for short), introduced by Godsil and McKay in \cite{Godsil1982}.
It is a special version of Seidel switching.

\begin{definition}[GM-switching \cite{Godsil1982}] \label{Def:GMswitch}
	Let \(G=(V,E)\) be a graph. Further, let \((C_1, \dots, C_k, D)\) be a partition of \(V(G)\). Suppose that for all \(i, \, j \in \{1, \dots,k\}\):
	\begin{itemize}
		\item all vertices in \(C_i\) have an equal number of neighbours in \(C_j\),
		\item every vertex in \(D\) is adjacent to either \(0, \frac{n_i}{2} \text{ or } n_i\) vertices of \(C_i\), with \(n_i\) the number of vertices in \(C_i\).
	\end{itemize} 
	Take all vertices \(v\) in \(D\) that have \(\frac{n_i}{2}\) neighbours in \(C_i\), remove those  \(\frac{n_i}{2}\) edges, and connect \(v\) with the other \(\frac{n_i}{2}\) vertices of \(C_i\). This method of constructing graphs is \emph{GM-switching} and the considered partition is the \emph{switching partition}.
\end{definition}

\begin{theorem} \cite{Godsil1982}
	Let \(G'\) be the graph obtained from \(G\) by GM-switching, then \(G\) and \(G'\) are cospectral.
\end{theorem}
 
%%%%%%%%%%%%%%%%%%%%%%%%%%%%%%
\subsubsection{WQH-switching}
%%%%%%%%%%%%%%%%%%%%%%%%%%%%%%

Another switching method was introduced by Wang, Qiu and Hu (for short, WQH-switching) in \cite{wang2019cospectral}.

\begin{definition}[WQH-switching \cite{wang2019cospectral}] \label{Def:WQH-switch}
	Let \(G=(V,E)\) be a graph. Denote \(C(v)\) as the set of neighbours of vertex \(v \in V\) in the set \(C \subseteq V\). Let \(C_1,C_2\) be disjoint subsets of \(V\) such that:
	\begin{itemize}
		\item \(|C_1| = |C_2|\),
		\item there is a constant \(c\) such that for all \(i,j \in \{1,2\}\), \(i\neq j\) and every \(v \in C_i: |C_i(v)| - |C_j(v)| = c\), 
		\item for every vertex \(u \in V \setminus (C_1 \cup C_2)\) either:
		\begin{enumerate}
			\item \(|C_1(u)| = |C_1|\) and \(|C_2(u)| = 0\),\label{Def:WQH-switch(a)}
			\item \(|C_1(u)| = 0\) and \(|C_2(u)| = |C_2|\),\label{Def:WQH-switch(b)}
			\item or \(|C_1(u)| = |C_2(u)|\).
		\end{enumerate}
	\end{itemize}
	For every \(u \in C_1 \cup C_2\) and every \(v \in V\setminus (C_1 \cup C_2)\) for which \ref{Def:WQH-switch(a)} or \ref{Def:WQH-switch(b)} is true, switch the adjacency between \(u\) and \(v\). This is called \emph{WQH-switching}, and \(\{C_1,C_2\}\) is called a \emph{switching set} of \(G\).
\end{definition}

\begin{theorem}\cite{wang2019cospectral}
	Let \(G'\) be the graph obtained from graph \(G\) by WQH-switching, then \(G\) and \(G'\) are cospectral.
\end{theorem}

We should note that the above WQH-switching is known to hold in a more general form, see \cite{qiu2020theorem}.

\subsection{Switching 2-designs} \label{Chap:SwitchDesigns}
%%%%%%%%%%%%%%%%%%%%%%%%%%%%%%

Next we introduce two switching methods for designs. The first is switching an oval. 
The second switching is a tool to construct new 2-designs, and will be called CS-switching.

%%%%%%%%%%%%%%%%%%%%%%%%%%%%%%
\subsubsection{Switching an oval}\label{Sec:switchOval}
%%%%%%%%%%%%%%%%%%%%%%%%%%%%%%

In \cite{denniston1982enumeration}, Denniston introduced a switching to create new $2$-$(25,9,3)$ designs from symmetric $2$-$(25,9,3)$ designs that have an oval. 
In his paper Denniston uses the definition of an oval as generalized by Assmus and Van Lint in \cite{assmus1979ovals}. 
However, in this definition, the number of points in an oval can vary. 
Since switching an oval only works if the oval contains four points exactly, we will define ovals as having four points.

\begin{definition}\label{Def:oval}
	An \emph{oval} is a set of four points of which no three lie in a common block.
\end{definition}

Consider a symmetric $2$-$(25,9,3)$ design that has a set of four points $\{P,Q,R,S\}$ that form an oval.
There are nine distinct blocks going through $P$ and $Q,R$ or $S$, which are the only blocks that contain \(P\).

\begin{definition}[Switching an oval \cite{denniston1982enumeration}]\label{Def:switchOval}
	Let \(\D\) be a symmetric $2$-$(25,9,3)$ design with oval $\{P,Q,R,S\}$. Take the three blocks that contain both \(P\) and \(Q\) and the three blocks that contain \(R\) and \(S\). Change the incidences between these six blocks and the oval. This switching method is called \emph{switching an oval}.
\end{definition}

\begin{theorem}\label{Thm:switchOval}\cite{denniston1982enumeration}
	Let \(\D'\) be the result of switching an oval of a symmetric $2$-$(25,9,3)$ design \(\D\), then \(\D'\) is again a symmetric $2$-$(25,9,3)$) design.
\end{theorem}

%%%%%%%%%%%%%%%%%%%%%%%%%%%%%%
\subsubsection{CS-switching}\label{Sec:switch2designs}
%%%%%%%%%%%%%%%%%%%%%%%%%%%%%%

Crnkovi{\'c} and {\v{S}}vob \cite{crnkovic2022switching} introduced a switching method for 2-designs (for short, CS-switching), and applied it to certain Menon designs with 36, 64 and 100 points to construct new designs.

\begin{definition}[CS-switching \cite{crnkovic2022switching}]\label{Def:switch2design}
	Let \(\mathcal{D = (P,B},I)\) be a 2-design and let \(\mathcal{B}_1 \subset \mathcal{B}\) be a set of blocks such that there are \(\mathcal{P}_1,\mathcal{P}_2 \subset \mathcal{P}\) with the following properties:
	\begin{enumerate}
		\item $(P,B) \notin {I}$ for every $(P,B) \in \mathcal{P}_1 \times \mathcal{B}_1$,
		\item $(P,B) \in {I}$ for every $(P,B) \in \mathcal{P}_2 \times \mathcal{B}_1$,
		\item $|\{B\in \mathcal{B}_1 : (P,B) \in {I}\}| = |\{B\in \mathcal{B}_1 : (P,B) \notin {I}\}|$ for every $P \in \mathcal{P}\setminus(\mathcal{P}_1 \cup \mathcal{P}_2)$
	\end{enumerate}
	Then $\mathcal{B}_1$ is called a \emph{switching set} of $\mathcal{D}$.
	
	$\mathcal{D}_1 = (\mathcal{P, B}, I_1)$ is obtained from the 2-design $\D$ by CS-switching with respect to $\mathcal{B}_1$ if:
	\begin{enumerate}
		\item $\forall B \in \mathcal{B \setminus B}_1, P \in \mathcal{P}: (P,B) \in I_1 \iff (P,B) \in I$
		\item $\forall B \in \mathcal{B}_1, P \in \mathcal{P}_1 \cup \mathcal{P}_2:(P,B) \in I_1 \iff (P,B) \in I $
		\item $\forall B \in \mathcal{B}_1, P \in \mathcal{P} \setminus (\mathcal{P}_1 \cup \mathcal{P}_2):(P,B) \in I_1 \iff (P,B) \notin I $
	\end{enumerate}
\end{definition}

\begin{theorem}\label{Thm:CS-switch}\cite{crnkovic2022switching}
	The design \(\D_1\) derived from the $2$-$(v,k,\lambda)$ design \(\D\) by CS-switching with respect to \(\B_1\) is again a $2$-$(v,k,\lambda)$ design.
\end{theorem}

\section{CS-switching for divisible designs}
%%%%%%%%%%%%%%%%%%%%%%%%%%%%%%%%%

Notice that the proof of Theorem \ref{Thm:CS-switch} shows that two points of \(\D_1\) are together contained in \(\lambda\) blocks because it proves that the number of common blocks of two points does not change.
Therefore, the proof can be slightly altered to show that CS-switching also works for divisible designs, even though it was defined for 2-designs.
For completeness, we include the definition of CS-switching for divisible designs and prove that it again delivers a divisible design.

\begin{definition}[CS-switching (for divisible designs)]\label{Def:CS-switchDivDes}
	Let \(\mathcal{D = (P,B},I)\) be a divisible design and let \(\mathcal{B}_1 \subset \mathcal{B}\) be a set of blocks such that there are \(\mathcal{P}_1,\mathcal{P}_2 \subset \mathcal{P}\) with the following properties:
	\begin{enumerate}
		\item $(P,B) \notin {I}$ for every $(P,B) \in \mathcal{P}_1 \times \mathcal{B}_1$,
		\item $(P,B) \in {I}$ for every $(P,B) \in \mathcal{P}_2 \times \mathcal{B}_1$,
		\item $|\{B\in \mathcal{B}_1 : (P,B) \in {I}\}| = |\{B\in \mathcal{B}_1 : (P,B) \notin {I}\}|$ for every $P \in \mathcal{P}\setminus(\mathcal{P}_1 \cup \mathcal{P}_2)$
	\end{enumerate}
	Then $\mathcal{B}_1$ is called a \emph{switching set} of $\mathcal{D}$
	
	$\mathcal{D}_1 = (\mathcal{P, B}, I_1)$ is obtained from the divisible design $D$ by CS-switching with respect to $\mathcal{B}_1$ if:
	\begin{enumerate}
		\item $(P,B) \in I_1 \iff (P,B) \in I \; \forall B \in \mathcal{B \setminus B}_1, P \in \mathcal{P}$,
		\item $(P,B) \in I_1 \iff (P,B) \in I \; \forall B \in \mathcal{B}_1, P \in \mathcal{P}_1 \cup \mathcal{P}_2$,
		\item $(P,B) \in I_1 \iff (P,B) \notin I \; \forall B \in \mathcal{B}_1, P \in \mathcal{P} \setminus (\mathcal{P}_1 \cup \mathcal{P}_2)$.
	\end{enumerate}
\end{definition}

\begin{theorem} \label{Thm:CS-switchDivDes}
	The design \(\D_1\) derived from the divisible design \(\D\) with parameters \((v,k,\lambda_1,\lambda_2,n,m)\) by CS-switching with respect to \(\B_1\) is again a divisible design with parameters \((v,k,\lambda_1,\lambda_2,n,m)\).
\end{theorem}

\begin{proof}
	We first prove that every block in \(\D_1\) is incident with \(k\) points.
	Since incidence only changes for blocks in \(\B_1\), every block in \(\B\setminus\B_1\) still contains \(k\) points. 
	Let \(B\) be a block of \(\B_1\). This block is incident with \(k-|\P_2|\) points of \(\P\setminus(\P_1\cup\P_2)\) in \(\D\). Because of the third property of \(\B_1\), \[|\{(P,B) \in I : P\in \P\setminus(\P_1\cup\P_2), B \in \B_1\}| = \frac{|\P\setminus(\P_1\cup\P_2) \times \B_1|}{2}.\]
	Which means that \[k-|\P_2| = \frac{|\P\setminus(\P_1\cup\P_2)|}{2}.\]
	Thus, the block \(B\) is incident with \(k-|\P_2|\) points of \(\P\setminus(\P_1\cup\P_2)\) in both \(\D\) and in \(\D_1\). Therefore, all blocks of \(\B_1\) contain \(k\) points in \(\D_1\).
	
	Secondly, we show that the number of common blocks of two distinct points remains the same. Incidence is only changed for points in \(\P\setminus(\P_1\cup\P_2)\), so for every two points in \(\P_1\cup\P_2\) this is true.
	After switching, a point of \(\P\setminus(\P_1\cup\P_2)\) is still incident with half of the blocks of \(\B_1\). Since a point of \(\P_1\cup\P_2\) is adjacent with either all of no blocks of \(\B_1\), a point of \(\P_1\cup\P_2\) and a point of \(\P\setminus(\P_1\cup\P_2)\) remain incident with the same number of common blocks after switching.
	Now let \(P_1\) and \(P_2\) be two points from \(\P\setminus(\P_1\cup\P_2)\). Denote by \(x\) the number of blocks that are incident with both \(P_1\) and \(P_2\) and by \(y\) the number of blocks that contain neither \(P_1\) nor \(P_2\). Because \(P_1\) and \(P_2\) are both contained in half of the blocks of \(\B_1\), \[|\B_1| = y + \frac{|\B_1|}{2} + \frac{|\B_1|}{2} -x.\]
	Thus \(y = x\), meaning that after switching, there will still be \(x\) blocks of \(\B_1\) that contain both \(P_1\) and \(P_2\). Since incidence with the other blocks in \(\B\setminus\B_1\) does not change, points \(P_1\) and \(P_2\) are incident with the same number of common blocks in \(\D\) and \(\D_1\). 
	Since the number of common blocks of two points never changes, the same \(n\) groups of size \(m\) can be selected in \(\D_1\) to get a divisible design with parameters \((v,k,\lambda_1,\lambda_2,n,m)\). 
\end{proof}

\begin{example}
	Consider the incidence matrix \(A\) of a divisible design found in Appendix \ref{Sec:DivDesHadMat}. This divisible design has parameters \((24,10,6,3,3,8)\), and $A$ is actually the adjacency matrix of a divisible design graph with the same parameters (see \cite[Construction 4.9]{haemers2011divisible}). Label the blocks of the divisible design by \(\B = \{B_1,B_2,\dots,B_{24}\}\) and the points by \(\P = \{P_1,P_2,\dots,P_{24}\}\) in order of the incidence matrix \(A\). Further, let \(\B_1 = \{B_3,B_4,B_7,B_8\}\), \(\P_1 = \{P_1,P_2,P_5,P_6\}\) and \(\P_2 = \{P_{17},P_{18},\dots,P_{24}\}\). The incidence matrix \(A'\) of the divisible design after CS-switching is also included in Appendix \ref{Sec:DivDesHadMat}. This is again a \((24,10,6,3,3,8)\) divisible design that is non-isomorphic to the divisible design with incidence matrix \(A\). The blocks of \(\B_1\) are indicated in grey.
\end{example}

 In Section \ref{sec:applications}, we will illustrate some applications of the switching method for obtaining divisible designs presented in Theorem \ref{Thm:CS-switchDivDes}.
%%%%%%%%%%%%%%%%%%%%%%%%%%%%%%%%%
\section{Linking graph switching to design switching}
%%%%%%%%%%%%%%%%%%%%%%%%%%%%%%%%%

%%%%%%%%%%%%%%%%%%%%%%%%%%%%%%%%%
\subsection{Switching the incidence graph of a Menon design with an oval}\label{Sec:GraphBT16}%%%%%%%%%%%%%%%%%%%%%%%%%%%%%%%%%

In \cite{orrick2008switching}, Orrick mentions a link between his switching a closed quadruple method and the switching an oval method for symmetric $(25,9,3)$ block designs, introduced by Denniston in \cite{denniston1982enumeration}.

Recall that for $2$-$(25,9,3)$ designs, an oval (see Definition \ref{Def:oval}) contains four points for which every block contains no or two of those points. When considering the incidence matrix of the design, this oval will correspond to four rows in which every column has either two or zero $1$s. 
Replace all $0$s in the incidence matrix by $-1$, and the rows of the oval satisfy the definition of a closed quadruple. This is also the case when we replace all $0$s with $1$s and all $1$s with $-1$s.
Denniston's switching an oval (see Definition \ref{Def:switchOval}) is now equivalent to switching a closed quadruple in the \((1,-1)\)-matrix related to the incidence matrix. 

Notice that the $(1,-1)$-matrix we get in this specific case is a square matrix of order 25, which cannot be a Hadamard matrix, but the same idea can be used for Menon designs, which do come from Hadamard matrices. 

First, the definition of switching an oval must be generalized to be able to use it on other designs. To do this we need an extra condition for the oval. It is namely not enough that no three points are incident with the same block. The proof that switching an oval again delivers a symmetric 2-design uses the fact that every block contains no or two points of the oval.
This means that for every point in the oval, the number of blocks through that point and any other point of the oval equals the total number of blocks incident with that point. This means \(3\lambda = k\).

\begin{definition}[Switching an oval (generalized)]
	Let \(\D\) be a symmetric $2$-$(v,3\lambda,\lambda)$
     design that contains an oval \(\{P,Q,R,S\}\). Take all blocks that are incident with both \(P\) and \(Q\), and all blocks that are incident with both \(R\) and \(S\). Switch adjacency between these blocks and the oval. This is a generalization of \emph{switching an oval}.
\end{definition}

Now we are ready to state the next result, whose proof is similar to the one of Theorem \ref{Thm:switchOval}, but for general parameters.

\begin{theorem}
	The result of switching an oval of a symmetric $2$-$(v,3\lambda,\lambda)$ design is again a symmetric $2$-$(v,3\lambda,\lambda)$ design.
\end{theorem}

\begin{proof}
	It is clear that every block still has \(k\) points after switching and the number of points and blocks doesn't change.
	We just need to check if every two blocks intersect in $\lambda$ points.
	Incidence only changes for the $2\lambda$ blocks that contain \(P\) and \(Q\) or \(R\) and \(S\).
	Let \(\B_1\) be the set of these $2\lambda$ blocks and \(B_1, B_2 \in \B_1\).
	If \(B_1\) contains \(P\) and \(Q\) and \(B_2\) contains \(R\) and \(S\), then the intersection remains the same after switching.
	If \(B_1\) and \(B_2\) both contain the same points of the oval, the intersection loses two points, but also gains two points after switching.
	Now consider \(B_1\) to be a block of \(\B\) and \(B_2\) not.
	If \(B_2\) contains no points of the oval the intersection with \(B_1\) is left unchanged.
	If \(B_2\) is incident with two points of the oval, without loss of generality we can say \(B_2\) contains \(P\) and \(R\) and \(B_1\) contains \(P\) and \(Q\). Then the intersection looses the point \(P\), but gains the point \(R\). Therefore the size of the intersection remains the same.
\end{proof}

For Bush-type Hadamard matrices of order 16 specifically, switching a closed quadruple, switching an oval and GM-switching can be related.
Indeed, take for example the Bush-type Hadamard matrix \(H\) \eqref{Eq:BT16} of order 16. The Menon design from \(H\) has parameters $(16,6,2)$. In the incidence graph, point \(i\) and block \(j\) are adjacent iff \(H_{ij} = -1\).
The first four rows form a \emph{closed quadruple} of the matrix, and therefore also an oval of the corresponding Menon design. Let \(D\) be the set of vertices corresponding to the first four rows. Choose \(C_1, \dots, C_9\) as indicated in \eqref{Eq:BT16}. One can easily check that \(\{D,C_1, \dots, C_9\}\) is a GM-switching partition of the incidence graph. GM-switching in this example is equivalent to switching a closed quadruple in the second field of the Hadamard matrix, and to switching an oval of the Menon design.
\\
{\small{
\begin{equation}\label{Eq:BT16}
	H=
	\begin{pNiceMatrix}
		1&1&1&1&-&1&-&1&-&-&1&1&1&-&-&1\\
		1&1&1&1&1&-&1&-&-&-&1&1&-&1&1&-\\
		1&1&1&1&-&1&-&1&1&1&-&-&-&1&1&-\\
		1&1&1&1&1&-&1&-&1&1&-&-&1&-&-&1\\
		\hline
		-&1&-&1&1&1&1&1&1&-&-&1&-&-&1&1\\
		1&-&1&-&1&1&1&1&-&1&1&-&-&-&1&1\\
		-&1&-&1&1&1&1&1&-&1&1&-&1&1&-&-\\
		1&-&1&-&1&1&1&1&1&-&-&1&1&1&-&-\\
		\hline
		-&-&1&1&1&-&-&1&1&1&1&1&-&1&-&1\\
		-&-&1&1&-&1&1&-&1&1&1&1&1&-&1&-\\
		1&1&-&-&1&-&-&1&1&1&1&1&1&-&1&-\\
		1&1&-&-&-&1&1&-&1&1&1&1&-&1&-&1\\
		\hline
		1&-&-&1&-&-&1&1&-&1&-&1&1&1&1&1\\
		-&1&1&-&-&-&1&1&1&-&1&-&1&1&1&1\\
		-&1&1&-&1&1&-&-&-&1&-&1&1&1&1&1\\
		1&-&-&1&1&1&-&-&1&-&1&-&1&1&1&1\\
		\CodeAfter
		\tikz \draw  (1-|5) -- (17-|5);
		\tikz \draw  (1-|9) -- (17-|9);
		\tikz \draw  (1-|13) -- (17-|13);
		\SubMatrix.{1-1}{4-16}\}[right-xshift=0.5em,name=A] 
		\tikz \node [right] at (A-right.east) {$D$} ;
		\SubMatrix.{5-1}{8-16}\}[right-xshift=0.5em,name=B] 
		\tikz \node [right] at (B-right.east) {$C_2$} ;
		\SubMatrix.{9-1}{12-16}\}[right-xshift=0.5em,name=C] 
		\tikz \node [right] at (C-right.east) {$C_3$} ;
		\SubMatrix.{13-1}{16-16}\}[right-xshift=0.5em,name=D] 
		\tikz \node [right] at (D-right.east) {$C_4$} ;
		\OverBrace[yshift=1.5mm,shorten]{1-1}{1-4}{C_5}
		\OverBrace[yshift=1.5mm,shorten]{1-5}{5-8}{C_1}
		\OverBrace[yshift=1.5mm,shorten]{1-9}{9-10}{C_6}
		\OverBrace[yshift=1.5mm,shorten]{1-11}{11-12}{C_7}
		\OverBrace[yshift=1.5mm,shorten]{1-14}{14-15}{C_9}
		\SubMatrix.{1-13}{1-13}.[name=E] 
		\tikz \node [above, yshift=5mm] at (E) {$C_8$} ;
		\SubMatrix.{1-16}{1-16}.[name=J] 
		\tikz \node [above, yshift=5mm] at (J) {$C_8$} ;
	\end{pNiceMatrix}   
\end{equation}
}}
However, we note that this result is not generalizable, since switching an oval requires \(3(n^2-n) = 2n^2-n\), thus \(n=2\).

%%%%%%%%%%%%%%%%%%%%%%%%%%%%%%%%%
\subsection{Switching incidence graphs of 2-designs}
%%%%%%%%%%%%%%%%%%%%%%%%%%%%%%%%%

%%%%%%%%%%%%%%%%%%%%%%%%%%%%%%%%%
\subsubsection{GM-switching}
%%%%%%%%%%%%%%%%%%%%%%%%%%%%%%%%%
Consider the CS-switching for 2-designs, see Definition \ref{Def:switch2design}. The set \(\{\B_1,\P\setminus(\P_1\cup\P_2),\B\setminus\B_1,\P_1,\P_2\}\) defines a partition of the vertex set of the incidence graph that satisfies some conditions:

\begin{enumerate}
	\item Every vertex in \(\B_1\) is adjacent to all, no or half of the vertices of \(\P\setminus(\P_1\cup\P_2), \B\setminus\B_1, \P_1, \P_2\).
	\item Every vertex of \(\P_1\) has zero neighbours in \(\P\setminus(\P_1\cup\P_2)\) and \(\P_2\), and \(r\) neighbours in \(\B\setminus\B_1\).
	\item Every vertex of \(\P_2\) has zero neighbours in \(\P\setminus(\P_1\cup\P_2)\) and \(\P_1\), and \(r-|\B_1|\) neighbours in \(\B\setminus\B_1\).
	\item Every vertex in \(\P\setminus(\P_1\cup\P_2)\) has zero neighbours in \(\P_1\) and \(\P_2\), and \(r-\frac{|\B_1|}{2}\) neighbours in \(\B\setminus\B_1\).
	\item Every vertex in \(\B\setminus\B_1\) has zero neighbours in \(\B_1\).
\end{enumerate}

These conditions align with the conditions needed for a partition of a vertex set to be a GM-switching partition. We only need the following extra conditions to be satisfied for \(\{\B_1,\P\setminus(\P_1\cup\P_2),\B\setminus\B_1,\P_1,\P_2\}\) to be a switching partition:

\begin{enumerate}
	\item Every vertex in \(\B\setminus\B_1\) has the same number of vertices in \(\P_1\).
	\item Every vertex in \(\B\setminus\B_1\) has the same number of vertices in \(\P_2\).
	\item Every vertex in \(\B\setminus\B_1\) has the same number of vertices in \(\P\setminus(\P_1\cup\P_2)\).
\end{enumerate}
Notice that since every vertex of \(\B\setminus\B_1\) has \(k\) neighbours in \(\P\), it suffices for two of the three conditions to be satisfied.

In the design this translates to the existence of two constants \(c_1\) and \(c_2\) such that every block in \(\B\setminus\B_1\) contains \(c_1\) points of \(\P_1\) and \(c_2\) points of \(\P_2\).
With this extra condition, a GM-switching partition can be found in the incidence graph of a \(2\)-design, such that GM-switching the graph with respect to this partition is equivalent to CS-switching the design. This is illustrated in the next result.

\begin{theorem}\label{Thm:2desSwitch=GM}
	Let \(\D = (\P,\B,I)\) be a 2-design and let \(\B_1 \subset \B\) such that there are \(\P_1,\P_2 \subset \P\) with the following properties:
	\begin{enumerate}
		\item $(P,B) \notin {I}$ for every $(P,B) \in \mathcal{P}_1 \times \mathcal{B}_1$,
		\item $(P,B) \in {I}$ for every $(P,B) \in \mathcal{P}_2 \times \mathcal{B}_1$,
		\item $|\{B\in \mathcal{B}_1 : (P,B) \in {I}\}| = |\{B\in \mathcal{B}_1 : (P,B) \notin {I}\}|$ for every $P \in \mathcal{P}\setminus(\mathcal{P}_1 \cup \mathcal{P}_2)$,
		\item \(\exists \, c_1: \; |\{P\in \P_1 : (P,B) \in {I}\}| = c_1\) for every \(B \in \B\setminus\B_1\),
		\item \(\exists \, c_2: \; |\{P\in \P_2 : (P,B) \in {I}\}| = c_2\) for every \(B \in \B\setminus\B_1\).
	\end{enumerate}
	Then \(\{\B_1,\P\setminus(\P_1\cup\P_2),\B\setminus\B_1,\P_1,\P_2\}\) is a GM-switching partition for the incidence graph \(IG(\D)\) and CS-switching \(\D\) with respect to the switching set \(\B_1\) is equivalent to GM-switching \(IG(\D)\) with respect to the switching partition \(\{\B_1,\P\setminus(\P_1\cup\P_2),\B\setminus\B_1,\P_1,\P_2\}\).
\end{theorem}

\begin{proof}
	As explained above, \(\{\B_1,\P\setminus(\P_1\cup\P_2),\B\setminus\B_1,\P_1,\P_2\}\) satisfies all conditions to be a GM-switching partition.
	When switching the 2-design, all incidences between \(\B_1\) and \(\P\setminus(\P_1\cup\P_2)\) are reversed.
	In the incidence graph this is equal to switching adjacencies between the vertices of \(\B_1\) and \(\P\setminus(\P_1\cup\P_2)\).
	When performing GM-switching on \(IG(\D)\) adjacency is reversed between \(\B_1\) and sets of \(\{\P\setminus(\P_1\cup\P_2),\B\setminus\B_1,\P_1,\P_2\}\) for which every vertex of \(\B_1\) is adjacent to half of the vertices of that set. This condition is only satisfied for set \(\P\setminus(\P_1\cup\P_2)\). Therefore, using GM-switching gives the same result as applying CS-switching on the 2-design.
\end{proof}

%%%%%%%%%%%%%%%%%%%%%%%%%%%%%%%%%
\subsubsection{WQH-switching}
%%%%%%%%%%%%%%%%%%%%%%%%%%%%%%%%%
Since in the previous section we have shown an equivalence between GM-switching and CS-switching using the incidence matrix, it is natural to investigate if a similar result can be obtained when considering WQH-switching (see Definition \ref{Def:WQH-switch}).

Assume we find a WQH-switching set \(\{C_1,C_2\}\). 
We will consider two possible cases: \(C_1\cup C_2 \subseteq \B\) and  \(C_1\cup C_2 \subseteq \P\).

%%%%%%%%%%%%%%%%%%%%%%%%%%%%%%%%%
\subsubsection*{Case 1: \(C_1\cup C_2 \subseteq \B\)}
%%%%%%%%%%%%%%%%%%%%%%%%%%%%%%%%%

By definition, we have 
\begin{itemize}
	\item \(|C_1| = |C_2|\).
	\item There is a constant \(c\) such that for all \(i,j \in \{1,2\}\), \(i\neq j\) and every \(v \in C_i: |C_i(v)| - |C_j(v)| = c\). 
	\item For every vertex \(u \in V \setminus (C_1 \cup C_2)\) either:
	\begin{enumerate}
		\item \(|C_1(u)| = |C_1|\) and \(|C_2(u)| = 0\),\label{Def:WQH-switch2(a)}
		\item \(|C_1(u)| = 0\) and \(|C_2(u)| = |C_2|\),\label{Def:WQH-switch2(b)}
		\item or \(|C_1(u)| = |C_2(u)|\).\label{Def:WQH-switch2(c)}
	\end{enumerate}
\end{itemize}

The constant \(c\) is zero, since \(C_1\cup C_2 \subseteq \B\) and the incidence graph is bipartite.
For the same reason, every \(u \in (V \setminus (C_1\cup C_2))\cap \B\) satisfies condition \eqref{Def:WQH-switch2(c)} with \(|C_1(u)|=|C_2(u)|=0\).

Let \(\P_1\) be the set of points that satisfy condition \eqref{Def:WQH-switch2(a)}, \(P_2\) the set of points satisfying condition \eqref{Def:WQH-switch2(b)} and \(\P_3\) the set of points that satisfy condition \eqref{Def:WQH-switch2(c)}.
The incidence matrix \(M\) then appears as follows:

\[M = \begin{pNiceMatrix}[first-row,last-col]
	C_1 & C_2 & & \\
	J & O & M_1& \P_1\\
	O & J & M_2 & \P_2\\
	N_1 & N_2 & C & \P_3
\end{pNiceMatrix},\]
where \(J\) is the all-one matrix, \(O\) is the all-zero matrix and for every row of \(\P_3\), the row sum in \(N_1\) is the same as in \(N_2\). Because every block contains \(k\) points and \(|\P_1| = |\P_2|\), every column of \(N_1\) and \(N_2\) has the same column sum.
After switching, we obtain the following incidence matrix \(M'\)
\[M' = \begin{pNiceMatrix}[first-row, last-col]
	C_1&C_2&&\\
	O & J & M_1 & \P_1\\
	J & O & M_2 & \P_2\\
	N_1 & N_2 & C & \P_3
\end{pNiceMatrix}\]

\medskip

\begin{theorem}\label{Thm:WQHSwitchDesigns}
	Let \(\{C_1,C_2\}\) be a WQH-switching set of the incidence graph \(IG(\D)\) of a $2$-$(v,k,\lambda)$ design \(\D\) with \(C_1\cup C_2 \subseteq \B\). Let \(\P_1, \P_2, \P_3\) be as defined above. If \(|\P_1| = |\P_2|\), then the graph obtained by WQH-switching \(IG(\D)\) is again the incidence graph of a $2$-$(v,k,\lambda)$ design \(\D'\).
\end{theorem}

\begin{proof}
	Let \(G'\) be the graph obtained after WQH-switching \(IG(\D)\). Let \(\P\) be the vertex set corresponding to the points of the 2-design and \(\B\) the vertex set corresponding to the blocks. Let \(\D'\) be the incidence structure obtained by taking the vertices of \(\P\) in \(G'\) as points and the vertices of \(\B\) in \(G'\) as blocks. A point \(P\) and a block \(B\) are incident iff the corresponding vertices \(P\) and \(B\) are adjacent.
	
	The number of vertices in \(\P\) remains the same, so the number of points is still \(v\).
	Because \(|\P_1| = |\P_2|\), the number of points in a given block remains \(k\).
	
	We prove that two points are still contained in \(\lambda\) common blocks.
	Let \(P_1,P_2\) be two vertices corresponding to two points \(P_1\) and \(P_2\). If \(P_1,P_2\) are both in \(\P_3\), then incidence does not change.
	If both are in \(\P_1\) or both are in \(\P_2\), then the two points lay in \(\lambda-|C_1|+|C_2|\), respectively \(\lambda - |C_2|+|C_1|\), common blocks. Since \(|C_1|=|C_2|\), this is still \(\lambda\).
	If, without loss of generality, \(P_1 \in \P_1\) and \(P_2\in \P_2\), then no blocks of \(C_1\cup C_2\) contain both \(P_1\) and \(P_2\) before or after switching, so the number of common blocks remains unchanged.
	Assume, without loss of generality, \(P_1 \in \P_i, \, i \in \{1,2\}\), and \(P_2 \in \P_3\). Because every vertex in \(\P_3\) satisfies condition \eqref{Def:WQH-switch2(c)} in the definition of WQH-switching, every point of \(\P_3\) lies in the same number of blocks of \(C_1\) as it does of \(C_2\). Therefore, the number of common blocks does not change after switching.
	The incidence structure \(\D'\) obtained from \(G'\) is thus again a $2$-$(v,k,\lambda)$ design.
\end{proof}

Using the previous result, we can now define a new switching method for 2-designs.

\begin{definition} \label{Def:WQHswitch2designs}
	Let \(\D=(\P,\B,I)\) be a 2-design. Let \(C_1, C_2 \subset \B\) be two sets of blocks with \(|C_1|=|C_2|\) such that there are \(\P_1,\P_2 \subset \P\) with the following properties:
	\begin{enumerate}
		\item \(|\P_1|=|\P_2|\),
		\item \((P,B) \in I\) for every \((P,B) \in \P_i\times C_i, \, i \in \{1,2\} \),
		\item \((P,B) \not\in I\) for every \((P,B) \in \P_i\times C_j, \, i,j \in \{1,2\}, i \neq j\),
		\item \(|\{B\in C_1: (P,B)\in I\}| = |\{B\in C_2:(P,B)\} \in I|\) for every \(P \in \P\setminus (\P_1\cup \P_2)\).
	\end{enumerate}
	Then \(\{C_1,C_2\}\) is a switching set of \(\D\).
	
	\(\D' = (\P,\B,I_s)\) is obtained from the 2-design \(\D\) by switching with respect to \(\{C_1,C_2\}\) if:
	\begin{enumerate}
		\item \((P,B) \in I_s\) for every \((P,B) \in \P_i \times C_j, \, i,j \in \{1,2\}, i \neq j\),
		\item \((P,B) \not\in I_s\) for every \((P,B) \in \P_i\times C_i, \, i \in \{1,2\} \),
		\item \((P,B) \in I_s \iff (P,B) \in I\) for every \((P,B) \in \P \times \B\setminus(C_1\cup C_2)\) and every \((P,B) \in \P\setminus (\P_1 \cup \P_2) \times \B\).
	\end{enumerate}
\end{definition}

\begin{theorem}
	Let \(\D\) be a $2$-$(v,k,\lambda)$ design and let \(\{C_1,C_2\}\) be a switching set of \(\D\). The design \(\D'\) derived from \(\D\) by switching with respect to \(\{C_1,C_2\}\) is again a $2$-$(v,k,\lambda)$ design.
\end{theorem}

\begin{proof}
	This follows directly from Theorem \ref{Thm:WQHSwitchDesigns}.
\end{proof}

\begin{example}
	Reconsider the Menon design from the Bush-type Hadamard matrix \(H\) of order 16, see  \eqref{Eq:BT16}. Let \(\{C_1,C_2\}\) be a switching partition of the Menon design as shown on the incidence matrix in \eqref{Eq:BT16MenonWQHswitch}. The incidence matrix of the design obtained by switching with respect to \(\{C_1,C_2\}\) is shown in \eqref{Eq:BT16MenonWQHswitchResult}. The obtained design is non-isomorphic to the original Menon design.\\

	\begin{equation}\label{Eq:BT16MenonWQHswitch}
		\begin{pNiceMatrix}
			0&0&0&0&1&0&1&0&\Block[fill=blue!15,rounded-corners]{4-4}{}1&1&0&0&0&1&1&0\\
			0&0&0&0&0&1&0&1&1&1&0&0&1&0&0&1\\
			\hline
			0&0&0&0&1&0&1&0&0&0&1&1&1&0&0&1\\
			0&0&0&0&0&1&0&1&0&0&1&1&0&1&1&0\\
			\hline
			1&0&1&0&0&0&0&0&0&1&1&0&1&1&0&0\\
			0&1&0&1&0&0&0&0&1&0&0&1&1&1&0&0\\
			1&0&1&0&0&0&0&0&1&0&0&1&0&0&1&1\\
			0&1&0&1&0&0&0&0&0&1&1&0&0&0&1&1\\
			1&1&0&0&0&1&1&0&0&0&0&0&1&0&1&0\\
			1&1&0&0&1&0&0&1&0&0&0&0&0&1&0&1\\
			0&0&1&1&0&1&1&0&0&0&0&0&0&1&0&1\\
			0&0&1&1&1&0&0&1&0&0&0&0&1&0&1&0\\
			0&1&1&0&1&1&0&0&1&0&1&0&0&0&0&0\\
			1&0&0&1&1&1&0&0&0&1&0&1&0&0&0&0\\
			1&0&0&1&0&0&1&1&1&0&1&0&0&0&0&0\\
			0&1&1&0&0&0&1&1&0&1&0&1&0&0&0&0\\
			\CodeAfter
			\tikz \draw  (1-|9) -- (17-|9);
			\tikz \draw  (1-|11) -- (17-|11);
			\tikz \draw  (1-|13) -- (17-|13);
			\SubMatrix.{1-1}{2-16}\}[right-xshift=0.5em,name=A] 
			\tikz \node [right] at (A-right.east) {$\P_1$} ;
			\SubMatrix.{3-1}{4-16}\}[right-xshift=0.5em,name=B] 
			\tikz \node [right] at (B-right.east) {$\P_2$} ;
			\OverBrace[yshift=1.5mm,shorten]{1-9}{1-10}{C_1}
			\OverBrace[yshift=1.5mm,shorten]{1-11}{1-12}{C_2}
		\end{pNiceMatrix}
	\end{equation}
	\\
	
	\begin{equation}\label{Eq:BT16MenonWQHswitchResult}
		\begin{pNiceMatrix}
			0&0&0&0&1&0&1&0&\Block[fill=blue!15,rounded-corners]{4-4}{}0&0&1&1&0&1&1&0\\
			0&0&0&0&0&1&0&1&0&0&1&1&1&0&0&1\\
			\hline
			0&0&0&0&1&0&1&0&1&1&0&0&1&0&0&1\\
			0&0&0&0&0&1&0&1&1&1&0&0&0&1&1&0\\
			\hline
			1&0&1&0&0&0&0&0&0&1&1&0&1&1&0&0\\
			0&1&0&1&0&0&0&0&1&0&0&1&1&1&0&0\\
			1&0&1&0&0&0&0&0&1&0&0&1&0&0&1&1\\
			0&1&0&1&0&0&0&0&0&1&1&0&0&0&1&1\\
			1&1&0&0&0&1&1&0&0&0&0&0&1&0&1&0\\
			1&1&0&0&1&0&0&1&0&0&0&0&0&1&0&1\\
			0&0&1&1&0&1&1&0&0&0&0&0&0&1&0&1\\
			0&0&1&1&1&0&0&1&0&0&0&0&1&0&1&0\\
			0&1&1&0&1&1&0&0&1&0&1&0&0&0&0&0\\
			1&0&0&1&1&1&0&0&0&1&0&1&0&0&0&0\\
			1&0&0&1&0&0&1&1&1&0&1&0&0&0&0&0\\
			0&1&1&0&0&0&1&1&0&1&0&1&0&0&0&0\\
			\CodeAfter
			\tikz \draw  (1-|9) -- (17-|9);
			\tikz \draw  (1-|11) -- (17-|11);
			\tikz \draw  (1-|13) -- (17-|13);
			\SubMatrix.{1-1}{2-16}\}[right-xshift=0.5em,name=A] 
			\tikz \node [right] at (A-right.east) {$\P_1$} ;
			\SubMatrix.{3-1}{4-16}\}[right-xshift=0.5em,name=B] 
			\tikz \node [right] at (B-right.east) {$\P_2$} ;
			\OverBrace[yshift=1.5mm,shorten]{1-9}{1-10}{C_1}
			\OverBrace[yshift=1.5mm,shorten]{1-11}{1-12}{C_2}
		\end{pNiceMatrix}
	\end{equation}
	
\end{example}

In \cite{ihringer2019new}, Ihringer and Munemasa defined a WQH-switching set for the block graph of $2$-$(v,k,1)$ designs. We show that this switching is in fact a special case of the switching we defined in Definition \ref{Def:WQHswitch2designs}, with \(|\P_1| = |\P_2| = 1\).

Ihringer and Munemasa's switching set is constructed as follows:
Let \(\D = (\P,\B)\) be a $2$-$(v,k,1)$ design and let \(\D' = (\P',\B')\) be a $2$-$(v',k,1)$ subdesign of \((\P,\B)\).
Take two points \(P_1,P_2 \in \P'\). Let \(C_1\) be all blocks in \(\B'\) that contain \(P_1\) but not \(P_2\), and \(C_2\) all blocks of \(\B'\) that contain \(P_2\) but not \(P_1\).

\begin{theorem}\cite{ihringer2019new}
	The pair \(\{C_1,C_2\}\) is a switching set for the block graph of \(\D\).
\end{theorem}

Although this switching set is defined for the block graph of a design, it is related to our switching of Definition \ref{Def:WQHswitch2designs}, as the following result shows.

\begin{theorem}
	Let \(\D = (\P,\B)\) be a $2$-$(v,k,1)$ design and let \(\D' = (\P',\B')\) be a $2$-$(v',k,1)$ subdesign of \((\P,\B)\).
	Take two points \(P_1,P_2 \in \P'\). Let \(C_1\) be all blocks in \(\B'\) that contain \(P_1\) but not \(P_2\), and \(C_2\) all blocks of \(\B'\) that contain \(P_2\) but not \(P_1\).
	Then \(\{C_1,C_2\}\) is a switching set of \(\D\) as defined in Definition \ref{Def:WQHswitch2designs}.
\end{theorem}

\begin{proof}
	Let \(\P_1 = \{P_1\}\) and \(\P_2 = \{P_2\}\). The conditions \((P,B) \in I\) for every \((P,B) \in \P_i\times C_i, \, i \in \{1,2\} \) and \((P,B) \not\in I\) for every \((P,B) \in \P_i\times C_j, \, i,j \in \{1,2\}, i \neq j\), are satisfied by definition.
	Let \(P \in \P'\setminus(\P_1\cup \P_2)\), thus \(P \in \P'\setminus\{P_1,P_2\}\). Since \(\D'\) is a \(2-(v',k,1)\) design, \(P\) and \(P_1\) lay together in exactly one block and \(P\) and \(P_2\) lay together in exactly one block. If \(P\) is incident with the unique block that contains both \(P_1\) and \(P_2\), then \(|\{B \in C_1:(P,B) \in I\}|=|\{B\in C_2:(P,B) \in I\}| = 0\). Otherwise \(|\{B \in C_1:(P,B) \in I\}|=|\{B\in C_2:(P,B) \in I\}| = 1\).
	If \(P \in \P\setminus\P'\) then \(|\{B \in C_1:(P,B) \in I\}|=|\{B\in C_2:(P,B) \in I\}| = 0\), because all \(k\) points of any block of \(C_1\cup C_2\) are in \(\P'\).
\end{proof}

%%%%%%%%%%%%%%%%%%%%%%%%%%%%%%%%%
\subsubsection*{Case 2: \(C_1\cup C_2 \subseteq \P\)}
%%%%%%%%%%%%%%%%%%%%%%%%%%%%%%%%%

If \(\{C_1,C_2\}\) is a switching set for the incidence graph \(IG(\D)\) of a \(2-(v,k,\lambda)\) design \(D\), then for every block \(B\), either 
\begin{itemize}
	\item \(C_1 \subseteq B\) and \(C_2 \cap B = \emptyset\),
	\item \(C_2 \subseteq B\) and \(C_1 \cap B = \emptyset\),
	\item or \(|B\cap C_1| = |B\cap C_2|\).
\end{itemize}

Let \(\B_1\) be the set of blocks that satisfy the first condition and \(\B_2\) the set of blocks that satisfy the second condition.

Similarly to case 1, we want \(IG(\D)\) to remain an incidence graph of a $2$-$(v,k,\lambda)$ design after switching. Therefore, we need at least the condition \(|\B_1|=|\B_2|\), otherwise the number of common blocks of two points of \(C_1\) would change after switching.

However, with this additional condition, \(\{\B_1,\B_2\}\) forms a switching set as defined in Definition \ref{Def:WQHswitch2designs} and all results of case 1 apply in case 2.

%%%%%%%%%%%%%%%%%%%%%%%%%%%%%%%%%%%%%%%%%%%%%
\section{Application: divisible designs that allow switching}\label{sec:applications}
%%%%%%%%%%%%%%%%%%%%%%%%%%%%%%%%%%%%%%%%%%%%%

In this section, we provide examples of divisible designs that allow the switching from Theorem \ref{Thm:CS-switchDivDes}.  First we give a construction of a family of divisible designs which uses skew-type Hadamard matrices and Hadamard 3-designs. Note that a Hadamard matrix $H$ of order $m$ is called a \emph{skew-type} if $H=A+I_m$, where $A^T=-A$ and $I_m$ is the identity matrix of order $m$.

%%%%%%%%%%%%%%%%%%%%%%%%%%%%%%%%%%%%%%%%%%%%%
\subsection{A family of divisible designs obtained from skew-type Hadamard matrices}
%%%%%%%%%%%%%%%%%%%%%%%%%%%%%%%%%%%%%%%%%%%%%

Let $H$ be a skew-type Hadamard matrix of order $4 \ell$ and let $D$ be a point-by-block incidence matrix of a Hadamard 3-design, which is a 2-design with parameters 2-$(2 \lambda +2, \lambda + 1, \lambda)$ (than $b=4 \lambda +2$ and $r=2 \lambda + 1)$. Replace each diagonal entry of $H$ by the $(2 \lambda +2) \times (4 \lambda +2)$ zero matrix, each off-diagonal entry value 1 of $H$ by $D$, and each entry value -1 of $H$ by $J-D$ (i.e. by the incidence matrix of the complement of the Hadamard design), where $J$ is the all-ones matrix of size $(2 \lambda +2) \times (4 \lambda +2)$ . The obtained matrix is an point-by-block incidence matrix of a divisible design with the following parameters
$$v=4 \ell (2 \lambda+2),  k=(4 \ell-1) (\lambda + 1),$$
$$\lambda_1=(4 \ell-1) \lambda,  \lambda_2=(2 \ell-1)(2 \lambda +1),$$
$$m=4 \ell, n=2 \lambda +2.$$

Every column of the matrix $H$ corresponds to a switching set of size $(4 \lambda +2)$ of the constructed DD. 

\begin{example}
Let the Hadamard matrix $H$ be given as follows
   \[
    \begin{pmatrix} 
        1 & 1 & 1 &  1\\
       -  & 1 & -  &  1\\
       -  & 1 & 1 &  -\\
       -  & -  & 1  &  1 
    \end{pmatrix},
    \]
and let the point-by-block incidence matrix $D$ of a Hadamard $3$-$(8,4,1)$ design, which is also a $2$-$(8,4,3)$ design, is given as
  \[
   \begin{pmatrix} 
    1 & 1 & 1 & 1 & 1 & 1 & 1 &  0 & 0 & 0 & 0 & 0 & 0 & 0\\
    0 & 0 & 0 & 1 & 0 & 1 & 1 &  1 & 1 & 1 & 0 & 1 & 0 & 0\\
    1 & 0 & 0 & 0 & 1 & 0 & 1 &  0 & 1 & 1 & 1 & 0 & 1 & 0\\
    1 & 1 & 0 & 0 & 0 & 1 & 0 &  0 & 0 & 1 & 1 & 1 & 0 & 1\\ 
    0 & 1 & 1 & 0 & 0 & 0 & 1 &  1 & 0 & 0 & 1 & 1 & 1 & 0\\
    1 & 0 & 1 & 1 & 0 & 0 & 0 &  0 & 1 & 0 & 0 & 1 & 1 & 1\\
    0 & 1 & 0 & 1 & 1 & 0 & 0 &  1 & 0 & 1 & 0 & 0 & 1 & 1\\
    0 & 0 & 1 & 0 & 1 & 1 & 0 &  1 & 1 & 0 & 1 & 0 & 0 & 1
   \end{pmatrix}.
  \]
The resulting divisible design has parameters $(24,12,9,7,4,8)$. The divisible design obtained by the above given construction has $4$ switching sets of size $14$, each corresponding to a column of $H$. 
\end{example}

\begin{remark}
When $\ell=(\lambda + 1)/2$, then the above given construction produce a 2-design, i.e. $\lambda_1=\lambda_2$.
\end{remark}

\begin{remark}
This construction also works in the case when $H$ is the skew-type Hadamard matrix of order 2. In that case, the obtained DD has parameters 
$v=2(2 \lambda + 2), k=\lambda + 1$, $\lambda_1=\lambda, \lambda_2=0$, $m=2$, $n=2 \lambda +2$. 
\end{remark}

\begin{remark}
Similarly, we can replace each diagonal entry of $H$ by the all-ones matrix $J$. In that case the obtained divisible design has different parameters $v=4 \ell (2 \lambda+2)$, $k=(4 \ell+1) (\lambda + 1)$, $\lambda_1=(4 \ell+3) \lambda + 2$,  $\lambda_2=(2 \ell+1)(2 \lambda +1)$, $m=4 \ell$, $n=2 \lambda +2$.
If $n=(\lambda + 1)/2$, this construction gives us a $2$-design.
\end{remark}

The divisible designs obtained by switching any of the divisible designs from the described family are isomorphic to the starting divisible design. However, below we give a divisible design with parameters (28,6,2,1,7,4) that admits switching which may produce a divisible design non-isomorphic to the starting divisible design. 

%%%%%%%%%%%%%%%%%%%%%%%%%%%%%%%%%%%%%%%%%%%%%
\subsection{Divisible design with parameters (28,6,2,1,7,4)}
%%%%%%%%%%%%%%%%%%%%%%%%%%%%%%%%%%%%%%%%%%%%%
Divisible design digraphs (DDDs) were introduced in \cite{crnkovic2015ddd}. Note that the neighborhood design of a divisible design digraph with parameters $(v,k, \lambda_1, \lambda_2,m,n)$ is a symmetric divisible design with the same parameters (see \cite[Theorem 1]{ crnkovic2015ddd}).

The DDD with parameters (28,6,2,1,7,4) given in \cite[Theorem 19]{ crnkovic2015ddd}, i.e. its neighborhood design, admits switching. This DDD is constructed in the following way. 
Let us define auxiliary matrices $u_1$, $u_2$ and $u_3$ as follows:

\begin{displaymath}  u_1=   \left(
\begin{tabular}{cccc}
1 & 1 & 0 & 0 \\
1 & 1 & 0 & 0 \\
0 & 0 & 1 & 1 \\
0 & 0 & 1 & 1
\end{tabular}                 \right),
\quad
u_2=   \left(
\begin{tabular}{cccc}
1 &  0  &  1 &  0 \\
0 &  1  &  0 &  1 \\
1 &  0  &  1 &  0 \\
0 &  1  &  0 &  1
\end{tabular}                 \right),
\quad
u_3=   \left(
\begin{tabular}{cccc}
1 &  0  &  0 &  1 \\
0 &  1  &  1 &  0 \\
0 &  1  &  1 &  0 \\
1 &  0  &  0 &  1
\end{tabular}                 \right).
\end{displaymath}
Then the circulant block matrix $D=circ(0,u_1,u_2,0,u_3,0,0)$ is the adjacency matrix of
a DDD with parameters (28,6,2,1,7,4).

If we consider the adjacency matrix of this DDD as the point-by-block incidence matrix of the divisible design, then each of the seven columns of blocks of the circulant  matrix $circ(0,u_1,u_2,0,u_3,0,0)$ is a switching set. Taking all combinations of these seven switching sets we get 128 divisible designs, including the starting divisible design, which split into two isomorphism classes.

%%%%%%%%%%%%%%%%%%%%%%%%%%%%%%%%%%%%%%%%%%%%%%%%%%%%%%
\subsection*{Acknowledgements} 
%%%%%%%%%%%%%%%%%%%%%%%%%%%%%%%%%%%%%%%%%%%%%%%%%%%%%%
A. Abiad is supported by the Dutch Research Council through the grants VI.Vidi.213.085 and OCENW.KLEIN.475. D. Crnkovi\' c and A. \v Svob are supported by Croatian Science Foundation under the project HRZZ-IP-2022-10-4571 and by European Union-NextGenerationEU, project number uniri-iz-25-46-KonGeoGraGru. 
The first and third author thank Leo Storme for inspiring discussions in an early stage of this work.

\bibliographystyle{amsplain}
\bibliography{sources}

\newpage
%%%%%%%%%%%%%%%%%%%%%%%%%%%%%%%%%%%%%%%%%%%%%
\section{Appendix}

%%%%%%%%%%%%%%%%%%%%%%%%%%%%%%%%%%%%%%%%%%%%%
\subsection{A divisible design graph from a Hadamard matrix}\label{Sec:DivDesHadMat}
%%%%%%%%%%%%%%%%%%%%%%%%%%%%%%%%%%%%%%%%%%%%%

The incidence matrices used in Section \ref{Sec:switch2designs} are included below.

\begin{equation*}
	A=\begin{pNiceMatrix}\Block[fill=red!15, rounded-corners]{24-8}{}
		0 & 1 & 1 & 1 & 0 & 1 & 1 & 1 &\Block[fill=yellow!15, rounded-corners]{24-8}{} 0 & 1 & 1 & 1 & 1 & 0 & 0 & 0 &\Block[fill=green!15, rounded-corners]{24-8}{} 0 & 0 & 0 & 0 & 0 & 0 & 0 & 0 \\
		1 & 0 & 1 & 1 & 1 & 0 & 1 & 1 & 1 & 0 & 1 & 1 & 0 & 1 & 0 & 0 & 0 & 0 & 0 & 0 & 0 & 0 & 0 & 0 \\
		1 & 1 & 0 & 1 & 1 & 1 & 0 & 1 & 1 & 1 & 0 & 1 & 0 & 0 & 1 & 0 & 0 & 0 & 0 & 0 & 0 & 0 & 0 & 0 \\
		1 & 1 & 1 & 0 & 1 & 1 & 1 & 0 & 1 & 1 & 1 & 0 & 0 & 0 & 0 & 1 & 0 & 0 & 0 & 0 & 0 & 0 & 0 & 0 \\
		0 & 1 & 1 & 1 & 0 & 1 & 1 & 1 & 1 & 0 & 0 & 0 & 0 & 1 & 1 & 1 & 0 & 0 & 0 & 0 & 0 & 0 & 0 & 0 \\
		1 & 0 & 1 & 1 & 1 & 0 & 1 & 1 & 0 & 1 & 0 & 0 & 1 & 0 & 1 & 1 & 0 & 0 & 0 & 0 & 0 & 0 & 0 & 0 \\
		1 & 1 & 0 & 1 & 1 & 1 & 0 & 1 & 0 & 0 & 1 & 0 & 1 & 1 & 0 & 1 & 0 & 0 & 0 & 0 & 0 & 0 & 0 & 0 \\
		1 & 1 & 1 & 0 & 1 & 1 & 1 & 0 & 0 & 0 & 0 & 1 & 1 & 1 & 1 & 0 & 0 & 0 & 0 & 0 & 0 & 0 & 0 & 0 \\
		\hline
		0 & 1 & 1 & 1 & 1 & 0 & 0 & 0 & 0 & 0 & 0 & 0 & 0 & 0 & 0 & 0 & 0 & 1 & 1 & 1 & 0 & 1 & 1 & 1 \\
		1 & 0 & 1 & 1 & 0 & 1 & 0 & 0 & 0 & 0 & 0 & 0 & 0 & 0 & 0 & 0 & 1 & 0 & 1 & 1 & 1 & 0 & 1 & 1 \\
		1 & 1 & 0 & 1 & 0 & 0 & 1 & 0 & 0 & 0 & 0 & 0 & 0 & 0 & 0 & 0 & 1 & 1 & 0 & 1 & 1 & 1 & 0 & 1 \\
		1 & 1 & 1 & 0 & 0 & 0 & 0 & 1 & 0 & 0 & 0 & 0 & 0 & 0 & 0 & 0 & 1 & 1 & 1 & 0 & 1 & 1 & 1 & 0 \\
		1 & 0 & 0 & 0 & 0 & 1 & 1 & 1 & 0 & 0 & 0 & 0 & 0 & 0 & 0 & 0 & 0 & 1 & 1 & 1 & 0 & 1 & 1 & 1 \\
		0 & 1 & 0 & 0 & 1 & 0 & 1 & 1 & 0 & 0 & 0 & 0 & 0 & 0 & 0 & 0 & 1 & 0 & 1 & 1 & 1 & 0 & 1 & 1 \\
		0 & 0 & 1 & 0 & 1 & 1 & 0 & 1 & 0 & 0 & 0 & 0 & 0 & 0 & 0 & 0 & 1 & 1 & 0 & 1 & 1 & 1 & 0 & 1 \\
		0 & 0 & 0 & 1 & 1 & 1 & 1 & 0 & 0 & 0 & 0 & 0 & 0 & 0 & 0 & 0 & 1 & 1 & 1 & 0 & 1 & 1 & 1 & 0 \\
		\hline
		0 & 0 & 0 & 0 & 0 & 0 & 0 & 0 & 0 & 1 & 1 & 1 & 0 & 1 & 1 & 1 & 0 & 1 & 1 & 1 & 1 & 0 & 0 & 0 \\
		0 & 0 & 0 & 0 & 0 & 0 & 0 & 0 & 1 & 0 & 1 & 1 & 1 & 0 & 1 & 1 & 1 & 0 & 1 & 1 & 0 & 1 & 0 & 0 \\
		0 & 0 & 0 & 0 & 0 & 0 & 0 & 0 & 1 & 1 & 0 & 1 & 1 & 1 & 0 & 1 & 1 & 1 & 0 & 1 & 0 & 0 & 1 & 0 \\
		0 & 0 & 0 & 0 & 0 & 0 & 0 & 0 & 1 & 1 & 1 & 0 & 1 & 1 & 1 & 0 & 1 & 1 & 1 & 0 & 0 & 0 & 0 & 1 \\
		0 & 0 & 0 & 0 & 0 & 0 & 0 & 0 & 0 & 1 & 1 & 1 & 0 & 1 & 1 & 1 & 1 & 0 & 0 & 0 & 0 & 1 & 1 & 1 \\
		0 & 0 & 0 & 0 & 0 & 0 & 0 & 0 & 1 & 0 & 1 & 1 & 1 & 0 & 1 & 1 & 0 & 1 & 0 & 0 & 1 & 0 & 1 & 1 \\
		0 & 0 & 0 & 0 & 0 & 0 & 0 & 0 & 1 & 1 & 0 & 1 & 1 & 1 & 0 & 1 & 0 & 0 & 1 & 0 & 1 & 1 & 0 & 1 \\
		0 & 0 & 0 & 0 & 0 & 0 & 0 & 0 & 1 & 1 & 1 & 0 & 1 & 1 & 1 & 0 & 0 & 0 & 0 & 1 & 1 & 1 & 1 & 0
		\CodeAfter
		\tikz \draw  (1-|9) -- (24-|9);
		\tikz \draw  (1-|17) -- (24-|17);
	\end{pNiceMatrix} 
\end{equation*}
\newpage
\begin{equation*}
	A'=\begin{pNiceMatrix}
		0 & 1 & \Block[fill=blue!15, rounded-corners]{24-2}{}1 & 1 & 0 & 1 & \Block[fill=blue!15, rounded-corners]{24-2}{} 1 & 1 & 0 & 1 & 1 & 1 & 1 & 0 & 0 & 0 & 0 & 0 & 0 & 0 & 0 & 0 & 0 & 0 \\
		1 & 0 & 1 & 1 & 1 & 0 & 1 & 1 & 1 & 0 & 1 & 1 & 0 & 1 & 0 & 0 & 0 & 0 & 0 & 0 & 0 & 0 & 0 & 0 \\
		1 & 1 & 1 & 0 & 1 & 1 & 1 & 0 & 1 & 1 & 0 & 1 & 0 & 0 & 1 & 0 & 0 & 0 & 0 & 0 & 0 & 0 & 0 & 0 \\
		1 & 1 & 0 & 1 & 1 & 1 & 0 & 1 & 1 & 1 & 1 & 0 & 0 & 0 & 0 & 1 & 0 & 0 & 0 & 0 & 0 & 0 & 0 & 0 \\
		0 & 1 & 1 & 1 & 0 & 1 & 1 & 1 & 1 & 0 & 0 & 0 & 0 & 1 & 1 & 1 & 0 & 0 & 0 & 0 & 0 & 0 & 0 & 0 \\
		1 & 0 & 1 & 1 & 1 & 0 & 1 & 1 & 0 & 1 & 0 & 0 & 1 & 0 & 1 & 1 & 0 & 0 & 0 & 0 & 0 & 0 & 0 & 0 \\
		1 & 1 & 1 & 0 & 1 & 1 & 1 & 0 & 0 & 0 & 1 & 0 & 1 & 1 & 0 & 1 & 0 & 0 & 0 & 0 & 0 & 0 & 0 & 0 \\
		1 & 1 & 0 & 1 & 1 & 1 & 0 & 1 & 0 & 0 & 0 & 1 & 1 & 1 & 1 & 0 & 0 & 0 & 0 & 0 & 0 & 0 & 0 & 0 \\
		\hline
		0 & 1 & 0 & 0 & 1 & 0 & 1 & 1 & 0 & 0 & 0 & 0 & 0 & 0 & 0 & 0 & 0 & 1 & 1 & 1 & 0 & 1 & 1 & 1 \\
		1 & 0 & 0 & 0 & 0 & 1 & 1 & 1 & 0 & 0 & 0 & 0 & 0 & 0 & 0 & 0 & 1 & 0 & 1 & 1 & 1 & 0 & 1 & 1 \\
		1 & 1 & 1 & 0 & 0 & 0 & 0 & 1 & 0 & 0 & 0 & 0 & 0 & 0 & 0 & 0 & 1 & 1 & 0 & 1 & 1 & 1 & 0 & 1 \\
		1 & 1 & 0 & 1 & 0 & 0 & 1 & 0 & 0 & 0 & 0 & 0 & 0 & 0 & 0 & 0 & 1 & 1 & 1 & 0 & 1 & 1 & 1 & 0 \\
		1 & 0 & 1 & 1 & 0 & 1 & 0 & 0 & 0 & 0 & 0 & 0 & 0 & 0 & 0 & 0 & 0 & 1 & 1 & 1 & 0 & 1 & 1 & 1 \\
		0 & 1 & 1 & 1 & 1 & 0 & 0 & 0 & 0 & 0 & 0 & 0 & 0 & 0 & 0 & 0 & 1 & 0 & 1 & 1 & 1 & 0 & 1 & 1 \\
		0 & 0 & 0 & 1 & 1 & 1 & 1 & 0 & 0 & 0 & 0 & 0 & 0 & 0 & 0 & 0 & 1 & 1 & 0 & 1 & 1 & 1 & 0 & 1 \\
		0 & 0 & 1 & 0 & 1 & 1 & 0 & 1 & 0 & 0 & 0 & 0 & 0 & 0 & 0 & 0 & 1 & 1 & 1 & 0 & 1 & 1 & 1 & 0 \\
		\hline
		0 & 0 & 0 & 0 & 0 & 0 & 0 & 0 & 0 & 1 & 1 & 1 & 0 & 1 & 1 & 1 & 0 & 1 & 1 & 1 & 1 & 0 & 0 & 0 \\
		0 & 0 & 0 & 0 & 0 & 0 & 0 & 0 & 1 & 0 & 1 & 1 & 1 & 0 & 1 & 1 & 1 & 0 & 1 & 1 & 0 & 1 & 0 & 0 \\
		0 & 0 & 0 & 0 & 0 & 0 & 0 & 0 & 1 & 1 & 0 & 1 & 1 & 1 & 0 & 1 & 1 & 1 & 0 & 1 & 0 & 0 & 1 & 0 \\
		0 & 0 & 0 & 0 & 0 & 0 & 0 & 0 & 1 & 1 & 1 & 0 & 1 & 1 & 1 & 0 & 1 & 1 & 1 & 0 & 0 & 0 & 0 & 1 \\
		0 & 0 & 0 & 0 & 0 & 0 & 0 & 0 & 0 & 1 & 1 & 1 & 0 & 1 & 1 & 1 & 1 & 0 & 0 & 0 & 0 & 1 & 1 & 1 \\
		0 & 0 & 0 & 0 & 0 & 0 & 0 & 0 & 1 & 0 & 1 & 1 & 1 & 0 & 1 & 1 & 0 & 1 & 0 & 0 & 1 & 0 & 1 & 1 \\
		0 & 0 & 0 & 0 & 0 & 0 & 0 & 0 & 1 & 1 & 0 & 1 & 1 & 1 & 0 & 1 & 0 & 0 & 1 & 0 & 1 & 1 & 0 & 1 \\
		0 & 0 & 0 & 0 & 0 & 0 & 0 & 0 & 1 & 1 & 1 & 0 & 1 & 1 & 1 & 0 & 0 & 0 & 0 & 1 & 1 & 1 & 1 & 0
		\CodeAfter
		\tikz \draw  (1-|9) -- (24-|9);
		\tikz \draw  (1-|17) -- (24-|17);
	\end{pNiceMatrix}
\end{equation*}

\end{document}